\documentclass[runningheads]{llncs}

\usepackage[T1]{fontenc}

\usepackage{amsmath}
\usepackage{amssymb}
\usepackage{verbatim}
\usepackage{xcolor}

\usepackage{graphicx}

\usepackage{stmaryrd}

\usepackage{mathrsfs} 

\usepackage{txfonts}  
  \DeclareSymbolFont{symbolsC}{U}{txsyc}{m}{n}
  \DeclareMathSymbol{\strictif}{\mathrel}{symbolsC}{74}
  \DeclareMathSymbol{\boxright}{\mathrel}{symbolsC}{128}

\usepackage{soul}

\usepackage{enumerate}

\usepackage{makecell}
\usepackage{rotating}

\newtheorem{df}{Definition}

\newtheorem{thm}[df]{Theorem}
\newtheorem{lem}[df]{Lemma}

\newtheorem{cor}[df]{Corollary}

\newtheorem{prop}[df]{Proposition}
\newtheorem{exm}[df]{Example}

\newtheorem{rmk}[df]{Remark}
\newtheorem{fact}[df]{Fact}

\newcommand{\lorbar}
{
  \mathbin{%
    \vcenter{\offinterlineskip
      \hbox to 6pt{\hss\rule{0.5em}{0.5pt}\hss}
      \kern0.2ex
       \hbox{$\vee$}
    }%
  }%
}

\newcommand{\upto}{%
  \mathrel{%
    \vcenter{\offinterlineskip
      \hbox{%
        \rule[0.35ex]{0.25em}{0.5pt}
        \kern-0.05em
        \rule[0.35ex]{0.25em}{0.5pt}
        \kern-0.1em
        \scalebox{0.6}[0.75]{$\prec$}
      }%
    }%
  }%
}

\newcommand{\ecf}{\rightarrowtail}
\newcommand{\necf}{\stackrel{.}{\rightarrowtail}}
\newcommand{\rcond}{\leadsto}

\newcommand{\maxcond}{\hookrightarrow}
\newcommand{\mincond}{\mathbin{\text{\rotatebox[origin=c]{180}{$\hookleftarrow$}}}}

\usepackage{trimclip}

\def\wcond{\,\hbox{$\circ$}\kern-1.5pt\hbox{\clipbox{2pt 0 0 0}{$\rightarrow$}}\,}

\def\bcond{\,\hbox{$\bullet$}\kern-1.5pt\hbox{\clipbox{2pt 0 0 0}{$\rightarrow$}}\,}

\newcommand{\nwcond}{{\dot\wcond}}

\newcommand{\PL}{\operatorname{PL}}
\newcommand{\La}{\mathcal{L}}

\newcommand{\NE}{\operatorname{NE}}

\newcommand{\SET}[1]{\mathbf{#1}}

\newcommand{\tand}{%
  \mathbin{%
    \vcenter{\offinterlineskip
      \hbox{$\wedge$}
      \kern0.2ex
      \hbox to 6pt{\hss\rule{0.5em}{0.5pt}\hss}
    }%
  }%
}

\newcommand{\wmimp}{%
  \mathbin{%
    \vcenter{\hbox{%
      \ooalign{%
        $\rightarrowtriangle$\cr
        \hspace{0.62em}\raisebox{0.02ex}{$\cdot$}\hfill\cr
        }
    }}%
  \!}%
}

\usepackage{pict2e,xfp}
\makeatletter
\newcommand{\vvee}{\mathbin{\mathpalette\d@vee\relax}}
\newcommand{\d@vee}[2]{%
  \begingroup
  \sbox\z@{$\m@th#1\vee$}%
  \setlength{\unitlength}{\ht\z@}%
  \kern0.1\wd\z@
  \begin{picture}(\fpeval{0.4\wd0/\unitlength},1)
  \roundcap\d@vee@thickness{#1}
  \Line(0,1)(\fpeval{0.4\wd0/\unitlength},0)
  \end{picture}%
  \kern-0.3\wd\z@\box\z@
  \endgroup
}
\newcommand{\d@vee@thickness}[1]{
  \linethickness{%
    1\fontdimen8
      \ifx#1\displaystyle\textfont\else
      \ifx#1\textstyle\textfont\else
      \ifx#1\scriptstyle\scriptfont\else
      \scriptscriptfont\fi\fi\fi 3
  }%
}
\makeatother

\newcommand{\ror}{
\mathbin{%
    \ooalign{%
      $\vee$\cr
\hidewidth\smash{\lower-0.3ex\hbox{$\cdot$}}\hidewidth\cr
    }%
  }%
  }

\newcommand{\vveedot}
 { \mathbin{%
    \vcenter{\offinterlineskip
      \hbox to 8.7pt{\hss\rule{0.5em}{0.5pt}\hss}      
      \kern0.2ex
      \hbox{$\vvee$}
    }%
  }%
}

\newcommand{\vveedots}{\mathbin{\mathpalette\vveedotaux\relax}}
\newcommand{\vveedotaux}[2]{%
  \vcenter{\offinterlineskip
    \hbox to 0.6em{\hss\vrule width 0.3em height 0.4pt\hss}%
    \kern0.15ex
    \hbox{$#1\vvee$}%
  }%
}

\newcommand{\arrowvec}[1]{\mathaccent"017E{#1}}

\newcommand{\commf}[1]{}
\newcommand{\commy}[1]{}

\newcommand{\todelete}[1]{
}
\newcommand{\tosimplydelete}[1]{
}

\newcommand{\corrf}[1]{#1}
\newcommand{\corry}[1]{#1}

\title{Possible and Impossible Conditionals for 
Team Logics
}

\author{Fausto Barbero\inst{1}\orcidID{0000-0002-0959-6977} \and Fan Yang\inst{2}\orcidID{0000-0003-0392-6522}}

\institute{Department of Philosophy, History and Art Studies, University of Helsinki \\ \email{fausto.barbero@helsinki.fi}  \and Department of Philosophy and Religious Studies, Utrecht University
}

\begin{document}

\maketitle

\begin{abstract}
We study whether a logic based on team semantics can be enriched with a conditional satisfying minimal requirements, 
\corrf{such as} preservation of the closure property of the logic, Modus Ponens, and the Deduction Theorem.  We show that such well-behaved conditionals exist for downward or upward closed logics, but do not typically exist for union closed, convex or intersection closed logics. We also briefly investigate conditionals  satisfying weaker requirements.


   
    \keywords{Team semantics \and Conditionals \and Deduction theorem \and Modus Ponens 
    }
\end{abstract}



In this paper, we study the possibility of having well-behaved conditionals in logics based on teams semantics. 
Team semantics is a generalization of  Tarskian semantics in which formulas are evaluated on \emph{teams}, which are \emph{sets} of evaluation points (such as  propositional valuations), rather than on single evaluation points. It was introduced by Hodges \cite{Hod1997} to characterise dependencies between first-order variables, in the context of independence-friendly logic (\cite{HinSan1989}, see also \cite{ManSanSev2011}). The framework has been developed further by V\"{a}\"{a}n\"{a}nen in dependence logic \cite{Vaa2007}, and inquisitive logic \cite{CiaRoe2011} \corrf{essentially uses the same semantics.} 
Many more team-based logics have then been introduced, including propositional \cite{YanVaa2016} and modal \cite{Vaa2008} dependence logic. These logics are extensions of classical logic, and thus inherit the classical implication $\alpha\to\beta:=\neg\alpha\vee\beta$ over classical formulas. The earliest versions of these logics (such as independence-friendly logic and dependence logic), however, do not have a conditional in their syntax for arbitrary formulas. In \cite{AbrVaa2009}, an adequate conditional $\rightarrow$ (known as \emph{intuitionistic implication}) was introduced for dependence logic, a language which has the 
{\em downward closure property}\commf{R1: Explain what this means}. 
The conditional was also used in inquisitive logic. While adding $\to$ to first-order dependence logic increases its expressive power significantly 
\cite{Yang2013}, the conditional is generally well behaved in downward closed logics, in the sense that it preserves downward closure, and satisfies both Modus Ponens and the Deduction Theorem. 


In recent years, many variants of dependence logic with different closure properties 
have received increasing attention, particularly  \emph{union closed} and \emph{convex} logics. 
None of the union closed languages in the literature has  a conditional (applied to arbitrary formulas) in their syntax, and 
the intuitionistic implication $\to$ clearly does not preserve union closure. While some of the convex logics do have 
$\to$ in their language (as $\to$ preserves convexity), 
the Deduction Theorem clearly fails in this context. 

In this paper, we raise the question whether a team logic with a certain closure property can have a conditional $>$ that is well-behaved in the sense that (1) $>$ preserves the closure property, (2) Modus Ponens and (3) Deduction Theorem hold for $>$ over the logic; namely, for 
any set $\{\varphi,\psi\}\cup \Gamma$ of formulas in the logic, \corrf{the following hold:} 
\begin{description}
\item[(Preservation)] If $\varphi$ and $\psi$ have the closure property, so does $\varphi>\psi$.
    \item[(Modus Ponens)]  $\varphi,\varphi>\psi \models \psi$.
    \item[(Deduction Theorem)] If $\Gamma,\varphi\models \psi$, then $\Gamma \models \varphi > \psi$.
\end{description}
Modus Ponens and the Deduction Theorem are of clear importance to ensure that a logic admits good axiomatic systems, so much so that e.g. the logicality of quantum logic was \corrf{at times} put in doubt 
(\corrf{see e.g.} \cite{JauPir1970}) 
\corrf{by noting} its lack of a conditional satisfying these constraints. In team semantics, we find it natural to add the requirement of Preservation, since the loss of a closure property has heavy repercussions\commf{R1: explain which repercussions (with refs?) [I have added a footnote]} on the validity of inference rules. 
We here address our question in the context of propositional languages, but many of our insights apply as well to first-order or modal team logics. 

The material implication $\rightarrowtriangle$ in team semantics satisfies both Modus Ponens and the Deduction Theorem, and, as we shall prove, it is the unique such conditional for \corrf{languages that meet certain expressivity requirements.} 
\corrf{Then, the fact that} 
$\rightarrowtriangle$  does not preserve e.g. union closure or convexity \corrf{can be used to show that} 
well-behaved conditionals do not exist for expressively complete languages with these closure properties. This impossibility result naturally leads us to consider weaker constraints, such as:
\begin{description}
    \item[(Feeble Modus Ponens)]  $\alpha,\alpha>\psi \models \psi$ for flat formulas $\alpha$.
    \item[(Feeble Deduction Theorem)] If $\Delta,\varphi\models \psi$, then $\Delta \models \varphi > \psi$ for sets $\Delta$ of flat formulas.
\end{description}
\corrf{We shall see that, in union closed and convex languages, impossibility results are obtained also when assuming some of the weaker constraints.}



In Section \ref{sec: conditional for upwards closed logic}, we prove that in downward closed languages, the intuitionistic implication $\to$ is the unique well-behaved conditional that satisfies all the requirements, and we also identify a unique well-behaved conditional for upward closed languages.




We open Section \ref{sec: impossibility} by showing a similar characterization of the uniquenss of material implication $\rightarrowtriangle$ and its weaker variant, from which our impossibility results for several non-downward closed, expressively complete languages follow. We then proceed to show that this result can be sharpened to cover \corrf{less} expressive union closed languages \corrf{and by weakening the requirement of the Deduction Theorem to its Feeble variant.} 
This argument is \corry{essentially}\commy{I strengthened this: this is essentially Hardegree's proof, as referee also pointed out.} 
Hardegree's proof of the impossibility of  well-behaved conditionals for quantum logic \cite{Har1974} \corry{recast in our setting}. The main idea  is to construct formulas for which the distributivity law $\varphi\wedge (\psi\vee\chi)\models(\varphi\wedge \psi)\vee (\varphi\wedge\chi)$ fails, and to derive a contradiction \corry{from the existence of such a well-behaved conditional.}

In Section \ref{sec: impossibility in convex logics}
, we adapt the argument to prove that for convex logics and intersection closed logics 
well-behaved conditional cannot in general exist either; \corrf{in the convex case, it suffices to assume\todelete{ either }Modus Ponens\todelete{ or the Deduction Theorem }to be in\todelete{ their}\corry{its} Feeble variant.} 
Our argument for convex logics uses the failure of a different type of distributivity law $\varphi\wedge( \psi\vveedot\chi) 
       \models (\varphi \land \psi) \vveedot (\varphi \land \chi)$, or in fact, the failure of a strengthened version of it, where $\vveedot$ is a variant of the global (or inquisitive) disjunction $\vvee$ as introduced in \cite{AntKnu2025}. For intersection closed logics, which currently lack a (complete) syntax in the literature, we work on an abstract level, and use the failure of the distributivity of conjunction over tensor conjunction $\tand$ from \cite{HelLuoVaa2024}. 

In Section 5, we consider some weak conditionals that can still be obtained, along the lines, e.g.,  suggested by van Fraassen \cite{Fra1973} in the context of quantum logic. \corrf{We use some of these conditionals to show} \corry{that the impossibility results are, in a sense, optimal.}

We end the paper in Section 6 with some  concluding remarks.

\section{Classical and Non-classical Languages}\label{sec:pre}


Let us start by defining the classical language in the team semantics setting. Let $Prop$ be a set of propositional variables; for simplicity, we assume that  $Prop$ is finite\footnote{\label{footnote: locality}Propositional team logics are more commonly defined using a countably infinite set of propositional variables. Since all of the languages considered in the paper enjoy the \emph{locality} property (i.e., whether a formula is satisfied only depends on the variables that actually occur in the formula), this restriction is not substantial.
}. 
The language $L_c$ of classical propositional logic is defined by the grammar 
\[\alpha::= p\in Prop\mid \bot\mid \neg \alpha\mid \alpha\wedge\alpha\mid \alpha\vee\alpha\mid \alpha\to\alpha\]
We shall use the first Greek letters $\alpha,\beta,\dots$ 
to denote classical formulas. Generic formulas in the extensions of $L_c$ to be discussed 
are denoted by $\varphi,\psi,\chi,\theta$, etc.




A \emph{valuation} 
(over $Prop$) is a function $v:Prop \rightarrow \{0,1\}$. A  \emph{team} $T$ is a set of valuations. The team semantics for 
$L_c$ is defined inductively as:
\begin{itemize}
\item $T\models p$ ~~iff~~ for all $v\in T$, $v(p)=1$.
\item $T\models \bot$ ~~iff~~ $T=\emptyset$.
    \item $T\models \neg\alpha$ ~~iff~~ for all $v\in T$, $\{v\}\not\models \alpha$.

    \item $T\models \alpha\land\beta$ ~~iff~~ $T\models \alpha$ and $T\models \beta$.

    \item $T\models \alpha \lor \beta$ ~~iff~~ there are $T_1,T_2$ with $T = T_1 \cup T_2$, $T_1 \models \alpha$ and $T_2\models \beta$.

    \item $T\models \alpha \rightarrow \beta$ ~~iff~~ for all $S\subseteq T$,  $S\models \alpha$ implies $S\models \beta$. 
\end{itemize}
We assert that $\Gamma\models \varphi$, if for all teams $T$, 
$T\models \Gamma$ implies $T\models\varphi$.\commy{I shortened this-just to save space, as we are desperate for space}

It can be easily verified that classical formulas have the empty team property, are downward and union closed, convex, and flat; these concepts are defined as:
\begin{itemize}
    \item $\varphi$ has the \emph{empty team property} or is \emph{empty team closed} if $\emptyset\models \varphi$.

    \item $\varphi$ is \emph{downward closed} if, whenever $T\models \varphi$ and $S\subseteq T$, we have $S\models \varphi$ 
    \item $\varphi$ is \emph{convex} if, whenever $S,T\models \varphi$ and $S \subseteq R\subseteq T$, we have $R\models \varphi$.

    \item $\varphi$ is \emph{union closed} if, whenever $T\models \varphi$ and $S\models \varphi$, we have $S\cup T\models \varphi$ 

    \item $\varphi$ is \emph{flat} \corrf{in case} $T\models \varphi$ iff for all $v\in T$, $\{v\}\models \varphi$.


\end{itemize}
Clearly, downward closure implies convexity; moreover, a formula is flat if and only if it is union closed, downward closed and  empty team closed.

By using flatness, one can easily show that 
all classical tautologies are valid in team semantics as well. In particular, we have that $\neg\alpha\equiv \alpha\to \bot$ and 
$\alpha\to\beta\equiv \neg\alpha\vee\beta$. This conditional $\to$ satisfies both Modus Ponens and the Deduction Theorem  over the classical language. Moreover, 
$ \alpha\land (\beta\lor\gamma) \models (\alpha \land \beta) \lor (\alpha \land \gamma)$ (distributivity) holds.

        


We also consider closure properties on their own, independently of any specific syntax, in the following sense: A {\em team proposition} $\mathbb{P}$ (over $Prop$) is a set of teams (over $Prop$). We say that $\mathbb{P}$ is {\em empty team closed} if $\emptyset\in \mathbb{P}$;  $\mathbb{P}$ is downward closed if, $T\in \mathbb{P}$ and $S\subseteq T$ imply $S\in \mathbb{P}$; and similarly for other closure properties. A classical formula $\alpha$ defines a team proposition 
\[\llbracket\alpha\rrbracket:=\{T\subseteq \{0,1\}^{Prop}: T\models\alpha\},\]\commy{definitely desperate, removed the centering} 
and it 
is clearly empty team closed, downward closed, union closed,  convex and flat. 

The classical language $L_c$ is {\em expressively complete} for flat team propositions, in the sense that
\begin{enumerate}[(1)]
    \item for every formula $\alpha\in L_c$, the team proposition $\llbracket \alpha\rrbracket$ defined by $\alpha$ is flat, and 
    \item every flat team proposition $\mathbb{P}$ can be defined by a formula $\alpha\in L_c$, i.e., $\mathbb{P}=\llbracket \alpha\rrbracket$.
\end{enumerate}
To see why item (2) above holds, observe first that for a flat team proposition $\mathbb{P}$, we have $\mathbb{P}=\mathcal P(\bigcup \mathbb{P})$. 
Put $T=\bigcup \mathbb{P}$. Let $Prop=\{p_1,\dots, p_n\}$, and let
\begin{equation}\label{cl-exp-compl-for}
    \alpha_T=\bigvee_{v\in T}(p_1^{v(p_1)}\wedge\dots\wedge p_n^{v(p_n)}),
\end{equation}
where $p_i^1:=p_i$ and $p_i^0=\neg p_i$. It can be easily verified that $S\models \alpha_T$ iff $S\subseteq T$ iff $S\in \mathcal{P}(T)=\mathbb{P}$. Thus, $\mathbb{P} = \llbracket \alpha_T \rrbracket$.


In this paper, we also consider the dual properties of downward  and union closure, namely upward closure and intersection closure: a team proposition $\mathbb{P}$ is said to be 
\begin{itemize}
    \item \emph{ upward closed} if $T\in \mathbb{P}$ and $T\subseteq S$ imply $S\in \mathbb{P}$.
        \item  \emph{ closed under intersection} if $S,T\in \mathbb{P}$ implies $S\cap T\in \mathbb{P}$.
\end{itemize}
These two closure properties arise naturally in our context, although they are less well studied in the team semantics literature. 
Upward closure clearly implies union closure, and downward closure implies intersection closure. The classical language is therefore closed under intersection, but it is clearly not upward closed (as, e.g., any propositional variable $p$ is not upward closed). For (non-trivial) examples of the two new closure properties,  consider two arbitrary distinct non-empty 
teams $T,S\subseteq \{0,1\}^{Prop}$ such that $S\not\subseteq T$, $T\not \subseteq S$ \corrf{and $S\cap T\neq \emptyset$}.\commf{Without this last requirement, the statement ``nor downward closed'' in the second example can fail!} The team proposition 
\[\mathbb{P}_1=\{R\subseteq \{0,1\}^{Prop} :R\supseteq T\text{ or }R\supseteq S\}\]
is upward closed (but not intersection closed), and $\mathbb{P}_2=\{T,S,T\cap S \}$ is  intersection closed (but neither upward closed nor downward closed).



 
 In the following, we will consider languages (or ``logics'') $[\circ_1,\dots,\circ_n]$ generated by connectives $\circ_1,\dots,\circ_n$, \corry{or extensions $L[\circ_1,\dots,\circ_n]$ of a language $L$ with the listed connectives}. This list may include some connectives of $L_c$ and other ones, such as those we will introduce in later sections.  If the formulas of $[\circ_1,\dots,\circ_n]$ are not all flat, we will say that this language is \emph{non-classical}. 

As with the usual (single-valuation based) classical language, the sets of connectives $\{\neg,\wedge\}$, $\{\neg,\wedge,\vee\}$ and $\{\bot,\wedge,\to\}$ are {\em functionally complete} for the classical language $L_c$, in the sense that all the other connectives can be defined in terms of those from each of these  sets. In particular, the formula (\ref{cl-exp-compl-for}) can also be written in these languages; thus, each of the languages $[\neg,\wedge]$, $[\neg,\wedge,\vee]$ and $[\bot,\wedge,\to]$ is expressively complete for flat properties.  When building a non-classical language, on the other hand, the choice of a classical basis matters: while all the connectives of $L_c$ preserve downward closure, it is not hard to see that union closure is preserved only by $\neg,\wedge,\vee$, convexity only by $\neg,\wedge,\to$, and intersection closure only by $\neg,\land$.

\section{Well-behaved Conditionals for Downward and Upward Closed Logics
}\label{sec: conditional for upwards closed logic}


We first consider downward and upward closed languages $L$, where conditionals that preserve the corresponding closure property and satisfy Modus Ponens and the Deduction Theorem can be found, and they are also unique if $L$ is sufficiently expressive.
\commy{the uniqueness proof uses this: see below}
\commy{Mention this also in the introduction of the paper}



We start with downward closed languages, which are well studied in the literature. One such example is the extension of the classical language $[\neg,\wedge,\vee]$ with the global disjunction $\vvee$, whose semantic is defined as
\begin{itemize}
   \item $T\models \varphi \vvee \psi$ ~~iff~~ $T\models \varphi$ or $T\models \psi$.
\end{itemize}
 More explicitly, formulas of the language $[\neg,\wedge,\vee,\vvee]$ are formed by:
\[
\varphi::= p\mid \neg\varphi
\mid  \varphi \land \varphi \mid \varphi \vee \varphi\mid  \varphi \vvee \varphi
\]
It was shown in \cite{YanVaa2016} that this language is expressively complete for empty team closed and downward closed team propositions (\cite{YanVaa2016} restricts negation to propositional variables only, which is not a substantial restriction for our purpose; cf. \cite{Yang2026}).

In downward closed languages, \corrf{the conditional $\to$ ({\em intuitionistic implication}) already defined in the previous section} is well-known to preserve downward closure, and satisfies Modus Ponens and the Deduction Theorem in downward closed languages. If, in addition, $L$ can express all flat team propositions (i.e., team propositions of the form $\mathcal{P}(T)=\{S:S\subseteq T\}$\commy{Here we would actually need to discuss why all flat team propositions are of this powerset form. The proof Proposition \ref{failure-distr} uses this as well. But maybe we leave this discussion to the journal version - no space here}), $\to$ is the unique such conditional, as shown in the next theorem: 

\begin{thm}\label{int-imp-uniqueness}
Let $L$ be a downward closed language that can express all flat team propositions, and $>$ a binary operator. Then $L[>]$ is downward closed and satisfies both Modus Ponens and the Deduction Theorem for $>$ iff $>$ is equivalent to  $\to$ over $L[>]$, meaning that $\varphi>\psi\equiv\varphi\to\psi$ for all formulas $\varphi,\psi$ in $L[>]$.\commy{I defined this here and this way -not completely sure about the definition: should actually be $\varphi,\psi$ in any language? Or, strictly speaking: $\llbracket \mathcal{T}>\mathcal{S}\rrbracket=\llbracket \mathcal{T}\to\mathcal{S}\rrbracket$ --yeah, but let's not go into that...}\commy{No, the current definition is correct. The more general one is not: the proof uses downward closure. All good}
\end{thm}
\begin{proof}
We only give the proof of the left to right direction. Suppose $L$ a downward closed language that can express all flat team propositions, and $L[>]$ is downward closed and satisfies both Modus Ponens and the Deduction Theorem for $>$. We show that $\varphi>\psi\equiv\varphi\to\psi$.

   Suppose $T\models \varphi >\psi$. Let $S\subseteq T$ be such that $S\models\varphi$. By downward closure of $L[>]$, $S\models \varphi >\psi$ as well. Then, Modus Ponens for $>$ gives $S\models \psi$. Thus, $T\models \varphi \rightarrow\psi$.

  Conversely, suppose $T\models \varphi \rightarrow\psi$. Consider the flat team proposition $\mathcal{P}(T)$. There exists a formula $\alpha_T$ in $L$ such that $\llbracket\alpha_T\rrbracket=\mathcal{P}(T)$ (cf.  formula (\ref{cl-exp-compl-for}) in Section \ref{sec:pre}). Note that $\alpha_T,\varphi\models\psi$. Indeed, for any team $S$ such that $S\models\alpha_T$ and $S\models\varphi$, we have $S\subseteq T$, which implies $S\models\psi$, as $T\models \varphi \rightarrow\psi$. Now, $\alpha_T,\varphi\models\psi$ gives $\alpha_T\models \varphi>\psi$ by the Deduction Theorem for $L[>]$. Lastly, since $T\models\alpha_T$, we conclude that $T\models\varphi>\psi$, as required.\qed
\end{proof}

Next, we turn to logics with the upward closure property, the dual of downward closure. In contrast to downward closure, upward closed languages are not yet well understood in the literature. To the best of our knowledge, the only systematic treatment to date appears in very recent work by  H\"aggblom \cite{Haggblom2026b}, where she introduced the following language that is expressively complete for upward closed team propositions: 
\[
\varphi::=\top  \mid  p  \mid \neg p  \mid \varphi \wedge \varphi\mid \varphi \vvee \varphi\mid \vcenter{\hbox{%
      \ooalign{%
        \rotatebox[origin=c]{180}{$\triangle$}\cr
        \hidewidth\raisebox{0.12ex}{$\cdot$}\hidewidth\cr
      }%
    }}%
    \varphi
\]
where the propositional variables adopt the so-called existential semantics (rather than the usual one, which clearly violates upward closure) and the operator $\vcenter{\hbox{%
      \ooalign{%
        \rotatebox[origin=c]{180}{$\triangle$}\cr
        \hidewidth\raisebox{0.12ex}{$\cdot$}\hidewidth\cr
      }%
    }}$ defined as:
\begin{itemize}
    \item $T\models p$ ~~iff~~ there is an $v\in T$ such that $v(p)=1$
    \item $T\models \vcenter{\hbox{%
      \ooalign{%
        \rotatebox[origin=c]{180}{$\triangle$}\cr
        \hidewidth\raisebox{0.12ex}{$\cdot$}\hidewidth\cr
      }%
    }}%
    \varphi$ ~~iff~~ there exists $v\in T$ such that $\{v\}\models\varphi$. 
\end{itemize}


In the upward closed context, let us introduce a new conditional $\upto$, 
the dual of intuitionistic implication, which has the following semantics:
\begin{itemize}
\item $T\models \varphi \upto \psi$ ~~iff~~ for all $S\supseteq T$, $S\models \varphi$ implies $S\models\psi$.
\end{itemize}
This conditional $\upto$ clearly preserves upward closure, and respects Modus Ponens and the Deduction Theorem relative to upward closed languages $L$. In addition, if $L$ can express all co-flat team propositions, i.e., all team propositions of the form $Up(T)=\{S:S\supseteq T\}$
,
then $\upto$ is the unique such conditional; 
\corrf{i.e., an analogue of} 
Theorem \ref{int-imp-uniqueness} holds for $\upto$ over upward closed languages that can express all co-flat team propositions. 
\commf{Removed mention of proof in Appendix. After al, we are not giving a proof there, just promising it exists.}
\commy{I took out the mention of \emph{strong transitivity, import} and \emph{export} from here - as it does not any more feel natural to mention these... May find a way to mention these somewhere else}
This can be shown by a dual argument to that in the proof of Theorem \ref{int-imp-uniqueness},
where in showing $T\models\varphi\upto \psi$ implies $T\models\varphi> \psi$, we now take a formula $\alpha_T^\ast$ in $L$ such that $\llbracket\alpha_T^\ast\rrbracket=Up(T)=\{S:S\supseteq T\}$.

\section{Impossibility of  Well-behaved Conditionals for Union Closed Logics}\label{sec: impossibility}
\commy{OK, I basically aim to rewrite the text of the whole section, so I will not indicate my changes in the section separately}

Having seen two types of logics with well-behaved conditionals, in the rest of the paper we prove that such 
conditionals do not exist for non-classical logics with the other closure properties we consider. In this section, we first provide an argument that covers all languages that are \emph{expressively complete} for these closure properties. Since many known logics in the literature are expressively complete, this suggests that the requirements we imposed on well-behaved conditionals may need to be weakened. 
Yet, we shall show 
that the impossibility results persist for \corry{these languages even under weaker requirements. We prove this sharper impossibility result for union closed languages in this section; languages with other closure properties are treated in the next section.}



\corrf{Let us} first observe that if we drop the Preservation requirement, there is a conditional in the literature that satisfies both Modus Ponens and 
the Deduction Theorem, namely the material implication $\rightarrowtriangle$:
\begin{itemize}
    \item $T\models \varphi \rightarrowtriangle \psi$ iff $T\models \varphi$ implies $T\models \psi$. \hspace{10pt} (material implication)
\end{itemize}
Material implication, however, does not preserve any of the closure properties we consider (see Appendix \ref{Appendix-material-imp} for counterexamples). Among these, preservation of the empty team closure is 
easily recovered by considering a 
variant of the material implication:
\begin{itemize}
        \item $T\models \varphi \wmimp \psi$ iff $T=\emptyset$ \corrf{or} $T\models \varphi$ implies $T\models \psi$. \hspace{10pt} (weak material implication)
\end{itemize}


We now give a  characterization of these two versions of material implication that is similar to Theorem \ref{int-imp-uniqueness}. We  show that the (strong) material implication $\rightarrowtriangle$ is actually the unique conditional that satisfies the Modus Ponens and the Deduction Theorem, if the underlying language is sufficiently expressive\footnote{We are very grateful to an anonymous referee for suggesting this characterization and sketching the proof, which constitutes the crucial part of Theorem \ref{thm: reviewer's theorem}.\commy{good?}}, and similarly, its weaker version $\wmimp$ is the unique such for sufficiently expressive languages that are empty team closed.


\begin{thm}\label{thm: reviewer's theorem}
Let $L$ be a language that can express all team propositions of the  form $\{T\}$ for arbitrary $T$. Then $L[>]$  satisfies both Modus Ponens and the Deduction Theorem for $>$ iff $>$ is equivalent to the  material implication $\rightarrowtriangle$ over $L[>]$.

Similarly, let $L$ be an empty team closed language that can express all team propositions of the  form $\{\emptyset,T\}$ for arbitrary $T$. Then $L[>]$ is empty team closed and satisfies both Modus Ponens and the Deduction Theorem for $>$ iff $>$ is equivalent to the weak material implication $\wmimp$ over $L[>]$.
%

\end{thm}
\begin{proof}
   We give the detailed proof of the characterization for the weak material implication $\wmimp$; the proof for the strong one $\rightarrowtriangle$ is similar. We prove the nontrivial  left to right direction. Suppose $L[>]$ is empty team closed and satisfies both Modus Ponens and the Deduction Theorem. We prove that $\varphi>\psi\equiv \varphi\wmimp\psi$ for any formulas $\varphi,\psi$ in $L[>]$.\commy{Here I think the original result WAS Incorrect, because we have to take the formulas from $L[>]$.}

   Clearly, $\emptyset\models \varphi\wmimp\psi$ by definition, and $\emptyset\models \varphi>\psi$ as $L[>]$ is empty team closed. Now, take $T\neq\emptyset$.
   Suppose  $T\models \varphi > \psi$ and $T\models \varphi$. By Modus Ponens for $>$,  $T\models \psi$, giving $T\models \varphi \wmimp\psi$. Conversely, suppose $T\models \varphi \wmimp\psi$. Let $\chi_T$ be an $L$-formula such that $\llbracket\chi_T\rrbracket = \{\emptyset, T\}$. It suffices to show that $\chi_T, \varphi \models \psi$, as then we would obtain $\chi_T \models  \varphi > \psi$ by the Deduction Theorem for $>$, from which $T\models \varphi > \psi$ would follow as $T\models \chi_T$.
   
   Now, to show $\chi_T, \varphi \models \psi$, suppose  $S\models \chi_T$ and $S\models\varphi$. Then either $S=\emptyset$ or $S=T$. If $S=\emptyset$,  since $L[>]$ is empty team closed, we have $S\models \psi$. Otherwise, if $S=T$, since $S\models \varphi$ and $S\models \varphi \wmimp\psi$ we obtain $S\models \psi$ by Modus Ponens, as required.\qed
    
    \commy{Question/comment: Your original proof uses a set of formula $\Gamma_T$ instead of a single one $\chi_T$-- this can indeed work. Then the assumption is actually: $\{\emptyset,T\}$ can be defined by a set of $L$ formulas. ($L$ can express $\{\emptyset,T\}$ usually means that there is a single formula that defines the proposition)}\commf{Good point. I believe that, as long as we look at definability of $\{\emptyset, T\}$ for a $T$ with a finite variable domain (in propositional team semantics), the two notions of definability should coincide.}
\end{proof}
\commy{We should emphasize here that this argument relies heavily on the fact that L can express *all* those properties, while our arguments in the main part only need the existence of *3* formulas} 


Our impossibility result follows now as a corollary, since the weak and strong material implication clearly do not preserve union closure,  convexity, intersection closure, and some of their combinations (see Appendix \ref{Appendix-material-imp} for the proof).\commy{work in progress}

\begin{cor}\label{imp-res-general}
    If $L$ is expressively complete for  union and empty team  closed team propositions, then there is no binary operator $>$ such that $L[>]$ is closed under empty team and unions, and satisfies both Modus Ponens and the Deduction theorem for $>$.

\corry{Similarly for languages expressively complete for any of the following classes of team propositions: union closed, convex, intersection closed, convex and union closed, intersection and empty team closed, and convex and intersection closed.}\commy{Is this better?}   
\end{cor}
\begin{proof}
Suppose such a binary operator $>$ exists for a language $L$ that is expressively complete for empty team and union closed team propositions. Observe that each team proposition of the form $\{\emptyset,T\}$ is closed under empty team and unions, and $L$ can thus express all such team propositions. It then follows from Theorem \ref{thm: reviewer's theorem} that $>$ is equivalent to $\wmimp$, which leads to a contradiction as $\wmimp$ does not preserve union closure. 

The other cases follows similarly, by observing that each team proposition $\{T\}$ is convex and closed under unions and intersections, and additionally $\emptyset\in \{\emptyset, T\}$. 
\qed
\end{proof}

The above result applies to a number of known expressively complete logics in the literature, such as 
the extension of the classical language $[\neg,\land,\lor]$ with inclusion atoms $\arrowvec{a}\subseteq \arrowvec{b}$ with each $a_i,b_i\in Prop\cup\{\bot,\top\}$ (known as {\em propositional inclusion logic}), which is expressively complete for empty team  and union closed team propositions \cite{Yang2022}, where
\begin{itemize}
    \item $T\models \arrowvec a \subseteq \arrowvec b$  ~~iff~~ for each $v\in T$, there exists $u\in T$ such that $v(\arrowvec a) = u(\arrowvec b)$.  
\end{itemize}
However, Corollary \ref{imp-res-general} or Theorem \ref{thm: reviewer's theorem} does not apply, for example, to a weaker version of propositional inclusion logic with inclusion atoms restricted to the form $\arrowvec p\subseteq \arrowvec q$ (i.e., without occurrences of $\bot,\top$), 
as this weaker language is known (\cite{Yang2022}) not to be expressively complete, \corrf{and even flat in the one-variable case $Prop=\{p\}$. It cannot then} 
 express all propositions of the form $\{\emptyset, T\}$, 
\corrf{because e.g. the proposition $\{\emptyset, T\}$, where $T$ is the full team \corry{in one variable}  with valuations $u(p)=0$ and $v(p)=1$, is not flat.} 

In addition to considering weaker logics, a natural alternative for avoiding Corollary \ref{imp-res-general} or Theorem \ref{thm: reviewer's theorem} while still obtaining relatively well-behaved conditionals is to \corrf{impose restrictions on the kinds of} \corrf{formulas allowed in} 
Modus Ponens and Deduction Theorem \corrf{schemes}. 
Plausible as this approach may seem, we show in the rest of the paper that \corrf{several} impossibility results \corrf{are still obtained when replacing either Modus Ponens or the Deduction Theorem with their Feeble counterparts.} 

One crucial fact on which our argument relies is the failure of distributivity of conjunction $\wedge$ over disjunction $\vee$ in non-downward closed languages. We begin by proving a general characterisation of the distributive law. 




\begin{prop}\label{failure-distr}
Let $L$ be a language that can express all flat team propositions, and $\varphi$ a formula of $L$. 
Then $\varphi$ is downward closed iff
    $ \varphi\land (\psi\lor\chi) \models (\varphi \land \psi) \lor (\varphi \land \chi)$ holds for all formulas $\psi,\chi$ of $L$.
\end{prop}
\begin{proof}
The left to right direction is easy to verify. 
For the other direction, suppose $\varphi$ is not downward closed. Then there are teams $S, T$ such that $S\subseteq T$, $T\models\varphi$ and $S\not\models\varphi$. Consider the team propositions $\mathbb{P}_S=\mathcal{P}(S)$ and $\mathbb{P}_{T\setminus S}=\mathcal{P}(T\setminus S)$. These two team propositions are clearly flat (as, e.g., $\{v_1\},\dots,\{v_n\}\in \mathcal{P}(S)$ iff $\{v_1,\dots,v_n\}\subseteq S$ iff $\{v_1,\dots,v_n\}\in \mathcal{P}(S)$). Since $L$ 
can express all flat team propositions, there are \corry{(flat)}\commy{the proof of Thm 3.5 needs these flat} formulas $\alpha_{S}$ and $\alpha_{T\setminus S}$ in  $L$  
(cf. formula (\ref{cl-exp-compl-for}) in Section \ref{sec:pre}) such that $\mathcal{P}(S)=\llbracket\alpha_{S} \rrbracket$ and $\mathcal{P}(T\setminus S)=\llbracket\alpha_{T\setminus S} \rrbracket$. 
We have that $\varphi\wedge (\alpha_S\vee\alpha_{T\setminus S})\not\models (\varphi\wedge \alpha_S)\vee(\varphi\wedge\alpha_{T\setminus S})$. 

To see why, first note that $T\models \varphi$. Moreover,  $T\models \alpha_S\vee\alpha_{T\setminus S}$, since $T=S\cup (T\setminus S)$, $S\models\alpha_{S}$ (for $S\in \mathcal{P}(S)$) and $T\setminus S\models\alpha_{T\setminus S}$ (for $T\setminus S\in \mathcal{P}(T\setminus S)$). Thus, $T\models \varphi\wedge (\alpha_S\vee\alpha_{T\setminus S})$. 

On the other hand, if $T\models (\varphi\wedge \alpha_S)\vee(\varphi\wedge\alpha_{T\setminus S})$, then there are $T_1,T_2$ such that $T=T_1\cup T_2$, $T_1\models \varphi\wedge \alpha_S$ and $T_2\models \varphi\wedge\alpha_{T\setminus S}$. It follows that $T_1\subseteq S$ and $T_2\subseteq T\setminus S$. Since we also have $T=T_1\cup T_2$, it must be that $T_1=S$. But then $T_1\models \varphi$ would imply $S\models\varphi$, which contradicts our assumption.\qed
\end{proof}

Distributivity thus fails in any union closed language that properly extends the classical language $[\neg,\wedge,\vee]$ (such as propositional inclusion logic $[\neg,\wedge,\vee,\subseteq]$ or its weaker variant discussed above), as it contains non-downward closed formulas (e.g., $p \subseteq q$). 




Our argument for the impossibility result below also uses a simple characterization of union closure in terms of the idempotence of disjunction $\vee$: 
\begin{lem}\label{prop: inferences from closure properties}
        $\varphi$ is union closed iff $\varphi \lor \varphi \models \varphi$.

\end{lem}


\begin{proof}
We only give the proof for the left to right direction. 
Suppose $T\models \varphi \lor \varphi$. Then, there exist $R,S$ such that $T=R\cup S$, $R\models \varphi$ and $S\models\varphi$, which imply $T\models\varphi$, since $\varphi$ is union closed.
%
%
%
\qed\end{proof}

We are ready to prove another impossibility result, \corrf{which assumes the} Feeble Deduction Theorem. \corry{The main argument is due to Hardegree \cite{Har1974}.}\commy{Give Hardegree more credit}

\begin{thm}\label{thm: Hardegree impossibility}
For any union and empty team closed language $L$ that properly extends the classical language $[\neg,\wedge,\vee]$\commy{``(union closed) classical language''},  there is no binary operator $>$ such that $L[>]$ is union closed  and  satisfies both Modus Ponens and the Feeble Deduction Theorem for $>$.


%
\end{thm}

\begin{proof}
    Suppose that such an operator $>$ exists. Let $\varphi,\psi,\chi$ be formulas in $L$ for which $ \varphi\land (\psi\lor\chi) \not\models (\varphi \land \psi) \lor (\varphi \land \chi)$ with $\psi,\chi$ flat (such formulas exist by Proposition \ref{failure-distr} and its proof).  
%
Since $L$ is empty team closed, we have 
\begin{equation*}
    \varphi \land \psi \models (\varphi \land \psi) \lor (\varphi \land \chi) ~\text{ and }~    \varphi \land \chi \models (\varphi \land \psi) \lor (\varphi \land \chi),
    \end{equation*} 
which imply, by  the \corrf{Feeble} Deduction Theorem,  that
 \begin{equation*}\label{imp_main_thm_eq2}
    \psi \models \varphi > (\varphi \land \psi) \lor (\varphi \land \chi)  ~\text{ and }~    \chi \models \varphi > (\varphi \land \psi) \lor (\varphi \land \chi).
    \end{equation*} 
Since the $L[>]$-formula $\varphi > (\varphi \land \psi) \lor (\varphi \land \chi)$  is union closed, 
by Lemma \ref{prop: inferences from closure properties}, we derive
\begin{equation*}\label{imp_main_thm_eq3}
\psi \lor \chi \models (\varphi > (\varphi \land \psi) \lor (\varphi \land \chi)) \lor (\varphi > (\varphi \land \psi) \lor (\varphi \land \chi)) \models \varphi > (\varphi \land \psi) \lor (\varphi \land \chi).
    \end{equation*}
    Finally, by Modus Ponens, 
    $\varphi\land (\psi\lor\chi) \models (\varphi \land \psi) \lor (\varphi \land \chi)$, which is a contradiction.
 \qed\end{proof}

 We remark that the argument above actually depends only on the existence in $L$ of three formulas used in the counterexample for distributivity. Languages of this kind are in general not expressively complete. One example is the weaker version of propositional inclusion logic discussed earlier. \corrf{Another example is \emph{first-order inclusion logic} (\cite{Gal2012}), \corry{where Proposition \ref{failure-distr} does not apply but other counterexamples for distributivity can be found.}
\footnote{For example, 
\corry{without going into details:}
$x\subseteq y \land (P(x)\lor P(y)) \not\models (x\subseteq y \land P(x)) \lor (x\subseteq y \land P(y))$).}}

\section{Impossibility in Convex Logics and Intersection Closed Logics}\label{sec: impossibility in convex logics}
\commy{I combined two sections. Also, I will rewrite the text everywhere, and thus will not mark them}

In this section, we prove  \corrf{sharper} impossibility results 
for non-classical convex logics and intersection closed logics by adapting \corrf{the argument from} 
Theorem \ref{thm: Hardegree impossibility}.

%
%
 %

We first consider convex languages. A typical example, which is also expressively complete\commy{now I don't think we need to emphasize that these are expressively complete - thus I mention this only briefly here}, is that given in \cite{AntKnu2025}, which is obtained by extending the classical language $[\bot,\wedge,\to]$ with the outer global disjunction $\vveedot$ 
 and the epistemic might $\blacklozenge$, defined as
    \begin{itemize}
        \item $T\models \varphi\vveedot\psi$ ~~iff~~ there exists $S\supseteq T$ such that $S\models\varphi$ or $S\models\psi$.
        \item $T\models \blacklozenge\varphi$ ~~iff~~ there is a nonempty team $S\subseteq T$ such that $S\models\varphi$.
   \end{itemize}

We can characterize convexity  in terms of $\vveedot$ and  $\blacklozenge$: 


\begin{lem}\label{convexity_char_idem}
%
 If $\varphi$ is convex, then  $\blacklozenge\varphi,\varphi\vveedot\varphi\models\varphi$; the converse  holds if $\emptyset\not\models\varphi$.
\end{lem}


\begin{proof}

%
%

    If $\varphi$ is satisfied by no team, then $\varphi$ is convex and $\blacklozenge\varphi,\varphi\vveedot\varphi\models\varphi$ (as no team satisfies $\blacklozenge\varphi$ and $\varphi\vveedot\varphi$). 
    
Now, assume that $\varphi$ is satisfied by some teams.
Suppose $\varphi$ is convex. Suppose $T\models\blacklozenge\varphi$ and $T\models\varphi\vveedot\varphi$. Then there are $R,S$ such that $\emptyset\neq R\subseteq T\subseteq S$, $R\models\varphi$ and $S\models\varphi$. By convexity, we conclude $T\models\varphi$.

For the converse direction, suppose $\emptyset\not\models\varphi$, and suppose $\varphi$ is not convex. Then there are $R,S,T$ such that $R\subseteq T\subseteq S$, $R\models\varphi$, $S\models\varphi$ and $T\not\models\varphi$. Clearly, $T\models \varphi\vveedot\varphi$. 
By assumption, $R\neq \emptyset$, which then implies $T\models \blacklozenge\varphi$. This shows $\blacklozenge\varphi,\varphi\vveedot\varphi\not\models\varphi$.
\qed\end{proof}

 \begin{rmk}\commy{This can go to footnote, in case desperate}
     The above  can be easily strengthened to a full characterisation: $\varphi$ is convex iff $\Diamond\varphi,\varphi\vveedot\varphi\models\varphi$, where $\Diamond$ is a closely related variant of $\blacklozenge$, defined as
        $T\models \Diamond \varphi$ iff there is a team $S\subseteq T$ such that $S\models\varphi$.
 \end{rmk}


Similar to the union closed context, in convex logics, conjunction does not distribute over the outer global disjunction $\vveedot$, even under a stronger assumption; specifically, $\varphi\wedge( \psi\vveedot\chi),\blacklozenge\psi 
       \models (\varphi \land \psi) \vveedot (\varphi \land \chi)$ fails in general:

\begin{exm}\label{non-distr-conv}
\corry{We construct a counterexample with $\varphi$  classical. Consider the team $T= \{u,v\}$ with the valuations $u,v$ satisfying $u(p,q)=(1,1)$, $v(p,q)=(1,0)$. Let $\varphi=p$, $\psi= q$ and $\chi=\blacklozenge \neg p$. It is easy to verify that T 
satisfies the antecedents $p\wedge (q\vveedot \blacklozenge \neg p)$ and $\blacklozenge q$, but not the consequent $(p\wedge q)\vveedot(p\wedge \blacklozenge\neg p)$.}
\end{exm}

Next, we present our main impossibility theorem:



\begin{thm}\label{impossibility_convex} 
For any convex language $L$ that properly extends the (convex) classical language\commy{OK, now I'll leave it without a definition - no space to define it: since we said already $L$ is convex, connectives in its classical part have to preserve convexity then} with $\vveedot$ and $\blacklozenge$, there is no binary operator $>$ such that the language $L[>]$ is convex, and satisfies both Feeble Modus Ponens and the Deduction Theorem for $>$\todelete{, or both Modus Ponens and the Feeble Deduction Theorem for $>$}.\commy{does this work? Of course, define ``(convex) classical language'': we basically kind of mentioned this at the end of Section 1. There are two choices: $[\bot,\wedge,\to]$ or $[\neg,\wedge,\lorbar]$}
\end{thm}
\begin{proof}
\todelete{We first prove the result for Feeble Modus Ponens and the Deduction Theorem. }Suppose that such a conditional $>$ exists. 
Let $\varphi,\psi,\chi$ be $L$-formulas such that $\varphi, \psi\vveedot\chi,\blacklozenge\psi 
       \not\models (\varphi \land \psi) \vveedot (\varphi \land \chi)$ and $\varphi$ is flat (such as those given in Example \ref{non-distr-conv}(1), \corry{where the flat formula $\neg p$ can be expressed in $L$ with convexity-preserving connectives as it extends classical language}). Clearly,
\begin{equation*}\label{impossibility_convex_eq0}
\varphi, \psi \models (\varphi \land \psi) \vveedot (\varphi \land \chi)~\text{ and }~\varphi, \chi \models (\varphi \land \psi) \vveedot (\varphi \land \chi),
    \end{equation*}
which imply, by the Deduction Theorem, that
\begin{equation*}\label{impossibility_convex_eq1}
    \psi \models \varphi > (\varphi \land \psi) \vveedot (\varphi \land \chi)~\text{ and }~\chi \models \varphi > (\varphi \land \psi) \vveedot (\varphi \land \chi).
    \end{equation*}
    By the monotonicity of $\blacklozenge$ and $\vveedot$, it follows that 
    \begin{equation*}\label{impossibility_convex_eq2}
    \blacklozenge\psi  \models \blacklozenge\big(\varphi > (\varphi \land \psi) \vveedot (\varphi \land \chi)\big)
       \text{ and }
\psi \vveedot \chi \models (\varphi > (\varphi \land \psi) \vveedot (\varphi \land \chi)) \,\vveedot\, (\varphi > (\varphi \land \psi) \vveedot (\varphi \land \chi)).
    \end{equation*}
Since $>$ preserves convexity, the formula $\varphi > (\varphi \land \psi) \vveedot (\varphi \land \chi)$ is convex. Thus, by Lemma \ref{convexity_char_idem} 
we obtain that
$\psi\vveedot\chi,\blacklozenge\psi 
     \models  \varphi > (\varphi \land \psi) \vveedot (\varphi \land \chi)$.
By \corrf{Feeble} Modus Ponens, $\varphi, (\psi\vveedot\chi),\blacklozenge\psi 
     \models   (\varphi \land \psi) \vveedot (\varphi \land \chi)$, which is a contradiction.    
\qed\end{proof}

\vspace{0.5\baselineskip}

Next, we treat intersection closed logics.  
%
Unlike the other relatively well-understood closure properties  and their corresponding logics,
 logics with intersection closure property have remained largely unexplored in the literature. To the best of our knowledge, the topic appears to have been addressed only in very recent unpublished work by H\"{a}ggblom   \cite{Haggblom2026b} and by the first author  \cite{Barbero2026}. We therefore proceed in this section without assuming a (complete) syntax.
Our argument, however, makes use of the \emph{tensor conjunction} $\tand$ from \cite{HelLuoVaa2024}, defined as
\begin{itemize}
    \item $T\models \psi \tand \chi$ ~~iff~~ there are $R,S \supseteq T$ s. t. $T = R\cap S$, $R\models \psi$ and $S\models \chi$. 
\end{itemize}
It is easy to verify that the tensor conjunction $\tand$, as well as the usual conjunction $\wedge$, preserve intersection closure.

Intersection closure is characterised by the idempotence of $\tand$: 

\begin{lem}\label{lemma: characterization of intersection closure}
    $\varphi$ is intersection closed iff $\varphi \tand \varphi \models \varphi$.
\end{lem}


\begin{proof}
  Suppose $\varphi$ is intersection closed, and $T\models \varphi \tand \varphi$. Then, there are $R,S\models \varphi$ with $T = R\cap S$. Thus, $T\models \varphi$ by intersection closure.

Conversely, suppose $\varphi\tand \varphi \models \varphi$. If $R,S\models \varphi$, then $R \cap S \models \varphi \tand \varphi\models\varphi$. 
\qed\end{proof}

\begin{thm}\label{thm: intersection impossibility}
\corrf{Let $S$ and $T$ be such that $S\not\subseteq T$ and $T\not \subseteq S$.} If 
 $L$ is a language that can express $\{S\},\{T\}$ and $\{S\cap T\}$, then 
there is no binary operator $>$ such that  $L[>]$ is intersection closed and satisfies both Modus Ponens and the Deduction Theorem for $>$.
\end{thm}

\begin{proof}
Suppose such a conditional $>$ exists.  Let $S,T$ be two distinct teams such that $S\cap T \neq S,T$.  We then have
\[\theta_{S\cap T},\theta_S \models (\theta_{S\cap T} \land \theta_S) \tand (\theta_{S\cap T} \land \theta_T)~\text{ and }~\theta_{S\cap T},\theta_T \models (\theta_{S\cap T} \land \theta_S) \tand (\theta_{S\cap T} \land \theta_T),\]
as no team satisfies any of the two antecedents above. By Deduction Theorem, 
\[\theta_S \models \theta_{S\cap T}>(\theta_{S\cap T} \land \theta_S) \tand (\theta_{S\cap T} \land \theta_T)~\text{ and }~\theta_{T}\models \theta_{S\cap T} >(\theta_{S\cap T} \land \theta_S) \tand (\theta_{S\cap T} \land \theta_T).\]
Put $\varphi=\theta_{S\cap T}>(\theta_{S\cap T} \land \theta_S) \tand (\theta_{S\cap T} \land \theta_T)$, which is a formula closed under intersection, as $>$ preserves intersection closure. By Lemma \ref{lemma: characterization of intersection closure}, we have that
\(\theta_S\tand \theta_T\models \varphi\tand \varphi\models\varphi.\)
Then, by Modus Ponens, we arrive at
\(\theta_{S\cap T} \land (\theta_S \tand \theta_T) \models (\theta_{S\cap T} \land \theta_S) \tand (\theta_{S\cap T} \land \theta_T),\)
which is impossible, as $S\cap T$ satisfies the antecedent but not the consequent.
\qed\end{proof}


In contrast, languages that are both intersection  and upward closed admit a well-behaved conditional, namely $\upto$; see 
Appendix \ref{app-inter-up} for a proof.

Let us remark that, as in the union closed case, the proofs of Theorems \ref{impossibility_convex} and \ref{thm: intersection impossibility} rely only on the existence in $L$ of three formulas: either $\varphi,\psi,\chi$ witnessing the failure of the strengthened distributivity, as in the proof of the former theorem, or $\theta_S,\theta_T,\theta_{S\cap T}$ as in the latter.  The presence of the the operators $\blacklozenge$, $\vveedot$ or $\tand$ in the language $L$ is not substantial for the proofs either: For instance, the proof of Theorem \ref{impossibility_convex} requires only that $L$ 
contains formulas $\theta_\varphi,\theta_\psi,\theta_\chi$ in $L$  equivalent to the formulas $\varphi,\psi,\chi$ in Example \ref{non-distr-conv}, respectively, together with  formulas $\theta_{\psi\vveedots\chi}$, $\theta_{\blacklozenge\psi}$, $\theta_{(\varphi \land \psi) \vveedots (\varphi \land \chi)}$  equivalent to $\psi\vveedot\chi$, $\blacklozenge\psi$, and $(\varphi \land \psi) \vveedot (\varphi \land \chi)$, respectively.

\section{Weak Conditionals}\label{sec:weak-cond}


Having shown that certain classes of languages (union closed, convex, intersection closed) do not admit conditionals that are well-behaved with respect to stringent or even feeble requirements, a natural question is whether well-behaved conditionals can still exist in the remaining cases, or when we weaken the constraints further.
%
%
%

It has been claimed, e.g., by van Fraassen (\cite{Fra1973}) that the essential features that warrant a binary connective $>$ being considered  a conditional are Modus Ponens and \emph{Introduction Rule}, a special case of the Deduction Theorem:
\begin{center}
    If $\varphi\models \psi$, then $\models \varphi > \psi$.
\end{center}
In this regard, consider the following \emph{maximal implication} $\hookrightarrow$ (introduced in \cite{KonNur2009}) and its dual {\em minimal implication} $\mathbin{\text{\rotatebox[origin=c]{180}{$\hookleftarrow$}}}$: 
\begin{itemize}
\item  $T\models \varphi \hookrightarrow \psi$ ~~iff~~ for all \emph{maximal} $S\subseteq T$ such that $S\models\varphi$, we also have $S\models\psi$  
\item $T\models \varphi \mathbin{\text{\rotatebox[origin=c]{180}{$\hookleftarrow$}}} \psi$ ~~iff~~ for all \emph{minimal} $S\supseteq T$ such that $S\models\varphi$, we also have $S\models\psi$. 
\end{itemize}
The former 
is  well-behaved  in van Fraassen's sense for intersection closed languages, \corrf{and even satisfies the Feeble Deduction theorem; it then witnesses that we cannot weaken the Deduction Theorem assumption in the impossibility result for intersection closed logics (whether the Modus Ponens aasumption can be relaxed remains an open problem).} The latter 
is a van Fraassen conditional for union closed languages. 
%
Let us also remark that the \emph{linear implication} $\multimap$ from \cite{AbrVaa2009}, and the conditional $\rcond$ from \emph{relevance logic} (\cite{AndDunBel2017}) satisfy Modus Ponens and preserve union  and downward closure:
\begin{itemize}
    \item $T\models \varphi \multimap \psi$ ~~iff~~ for all  $S$ such that $S\models\varphi$, we also have $S\cup T\models\psi$ 

    \item $T\models \varphi \rcond \psi$ ~~iff~~ for all $S$ such that $S\models \varphi$, we have $T\cap S\models \psi$ 
\end{itemize}
 However,  the introduction rule fails for $\multimap$ in the worst possible way: $p\multimap p$ is not valid; similarly,  $\blacklozenge p \rcond \blacklozenge p$ is not valid either.


Alternatively, one may want to  retain the Deduction Theorem fully. The following \emph{epistemic indicative} $\wcond$ preserves union closure and satisfies the Deduction Theorem, as well as  Modus Ponens \emph{for downward closed antecedents}, but it lacks transitivity:
\begin{itemize}
\item $T\models \varphi \wcond \psi$ ~~iff~~  whenever  $S\models \varphi \text{ for all } S\subseteq T, \text{ we have } T\models \psi$
\end{itemize}
The existence of such a conditional shows \emph{a fortiori} that requiring Feeble Modus Ponens and the Deduction Theorem does not yield an impossibility in the union closed case, differently from what we have seen in the convex case.

Another approach yet is to consider conditionals that satisfy both Modus Ponens and the Deduction Theorem  partially. We propose two such conditionals here, both originated from an epistemic interpretation of teams (as descriptions of  uncertainty among various possibilities): 
\begin{itemize}
\item $T \models \varphi \ecf \psi$ ~~iff~~  whenever  $S\subseteq T$  and  $U\models\varphi$  for all  $U\subseteq S$,  we have  $S\models \psi$.
\item $T\models \varphi \bcond \psi$ ~~iff~~  whenever $S\models \varphi$ for all $S\subseteq T$, we have $S\models \psi$  for all $S\subseteq T$.
\end{itemize}
The former,
 which we shall call \emph{epistemic counterfactual}, expresses the fact that ``If we \emph{knew} that $\varphi$, then we \emph{would know} $\psi$'', while the latter,  
 called the \emph{epistemic conditional} expresses ``If we \emph{know} that $\varphi$, then we \emph{know} that $\psi$''. Both conditionals satisfy Modus Ponens \emph{for downward closed antecedents}, and the Deduction Theorem \emph{for downward closed contexts}. As for preservation properties, $\bcond$ preserves union and upward closure. The conditional $\ecf$ preserves union, upward, and intersection closure, and convexity, and thus it testifies that the requirement of Feeble Modus Ponus plus the Feeble Deduction Theorem does not lead to an impossibility in the case of \corrf{(union closed, intersection closed or) convex} languages. Interestingly, both $\ecf$ and the relevant conditional $\rcond$ coincide with the intuitionistic implication $\to$ over downward closed languages.

Tables 1 and 2 
in Section \ref{app-weak-cond} of the Appendix list some other typical desirable properties of conditionals that the above-mentioned ones satisfy or lack, such as (strong) transitivity,  antecedent strengthening, importation, and exportation.

\section{Concluding Remarks}

In this paper, we have studied the possibility of having well-behaved conditionals in team logics with various closure properties, requiring that such conditionals preserve the closure property of the logic, and satisfy both Modus Ponens and the Deduction Theorem, or their weaker variants. We proved that the intuitionistic implication $\to$ is the unique well-behaved conditional for downward closed logics, and similarly for the dual of intuitionistic implication $\upto$ for upward closed logics. We obtained analogous characterizations for material implication $\rightarrowtriangle$ and its weaker variant $\wmimp$ as well, from which we concluded that well-behaved conditionals do not exist for sufficiently expressive union closed, convex or intersection logics. We also showed that well-behaved conditionals satisfying weaker requirements typically fail to exist for these logics either.


Having ascertained that conditionals satisfying even weaker sets of requirements exist, a natural follow-up question is whether, given a natural list of requirements,  such operators can be classified completely. A further question concerns how the addition of these ``weak conditionals'' affects the properties of a logic; e.g., what is the effect of adding a conditional that does not preserve flatness to the classical language? 


\commf{If we are really desperate for space I would cut out the paragraph with $\multimap$ and $\rcond$. In that case, remove the reference to $\rcond$ from the penultimate paragraph.}



\vspace{\baselineskip}

\ackname ~We thank Samson Abramsky, Aleksi Anttila, Matilda H\"{a}ggblom and S\o{}ren Knudstorp for helpful discussions, and an anonymous reviewer for very constructive suggestions. The first author's research was supported by the Research Council of Finland grant n. 349803.


\begin{center}
    \textbf{APPENDIX}
\end{center}

\appendix

\section{Lack of Closure Properties for (Weak) Material Implication}\label{Appendix-material-imp}

\begin{lem}\label{lemma: material implication}
Neither $\rightarrowtriangle$ nor $\wmimp$ preserves union or intersection closure, or convexity.
\end{lem}

\begin{proof}
Consider the formulas $p\lor q$ and $r$, which have all the listed closure properties. Let $u,v,w$ be valuations such that $u(p,q,r)= (0,0,0), v(p,q,r)= (1,0,0)$ and $w(p,q,r)= (0,0,1)$. Clearly, for $T=\{u,v\}$ and $S=\{v,w\}$, we have $T$ and $S$ satisfy both $p\lor q\rightarrowtriangle r$ and $p\lor q\wmimp r$, while $T\cap S=\{v\}$ satisfies neither, showing that none of the implications preserve intersection closure. Moreover, we have $\emptyset \models p\lor q\rightarrowtriangle r$ and $\emptyset \models p\lor q\wmimp r$. Since $\emptyset\subseteq \{v\}\subseteq T$, preservation of convexity fails for the two implications as well.
 
For the failure of preservation of union closure, consider the formulas  $\blacklozenge p \land q$ and $r$, which are union closed (and convex). Let $u,v$ be valuations such that  $u(p,q,r) = (1,1,1)$ and $v(p,q,r)=(0,1,0)$. Clearly, both $\{u\}$ and $\{v\}$ satisfy $\blacklozenge p \land q \rightarrowtriangle r$ and $ \blacklozenge p \land q \wmimp r$, whereas their union $\{u,v\}$ satisfies neither implication.\qed
\end{proof}
The first part of the proof above also shows that $\rightarrowtriangle$ or $\wmimp$ does not preserve the combination of convexity and intersection closure; the second part of the proof also shows that none of them preserves the combination of convexity and union closure either.

\section{Languages that Are both Intersection Closed and Upward Closed}\label{app-inter-up}

At the end of section \ref{sec: impossibility in convex logics} we mentioned the fact that languages that are closed both upward and under intersections do admit a well-behaved conditional, namely $\upto$. This follows from the fact that $\upto$ preserves the \emph{conjunction}\commf{Note that in the previous section we say ``combination''. Which is better?} of intersection closure and upward closure (even though it does not preserve intersection closure alone).

\begin{prop} 
    If $\varphi$ is upward closed and $\psi$ is intersection closed, then $\varphi \upto \psi$ is upward and intersection closed.
In particular, $\upto$ preserves the property of being both upward and intersection closed.  
\end{prop}

\begin{proof}
It follows immediately from the semantics that the formula $\varphi \upto \psi$ is always upward closed. For intersection closure, suppose $T,S\models \varphi \upto \psi$.
   Let $R\supseteq T \cap S$ be such that $R\models\varphi$; we must show that $R\models \psi$. By upward closure of $\varphi$, $R\cup T\models \varphi$ and $R\cup S\models \varphi$. Thus, since $T\models \varphi \upto \psi$ (resp. $S\models \varphi \upto \psi$) we obtain $R\cup T\models \psi$ (resp. $R\cup S\models \psi$). But, since $R\supseteq T \cap S$, we have $(R\cup T)\cap(R\cup S) = R\cup (T\cap S) = R$; thus, by intersection closure of $\psi$, $R\models \psi$. \qed  
\end{proof}

\noindent 
We already know that $\upto$ satisfies Modus Ponens and the Deduction Theorem relative to upward closed languages, thus a fortiori in languages that are both  upward and intersection closed; thus, $\upto$ is well-behaved for this class of languages.

\section{Summary of the Weak Conditionals}\label{app-weak-cond}

Table 1 and Table 2 summarise the main properties of some of the most interesting weak conditionals identified in Section \ref{sec:weak-cond}, where the relevant propertiess are defined as:
\begin{description}
    \item[(Strong) transitivity:] $\varphi > \psi, \psi > \chi \models \varphi > \chi$.

    \item[Weak transitivity:] If $\models \varphi > \psi$ and $\models \psi > \chi$, then $\models \varphi > \chi$.
    \item[Intermediate transitivity:] If $\models \varphi>\psi$,   then $\psi > \chi \models \varphi>\chi$.
\item[Antecedent strengthening:] $\varphi >\psi \models (\varphi\land \chi) > \psi$.

\item[Importation:] $\varphi > (\psi > \chi) \models (\varphi \land \psi) > \chi$.

\item[Exportation:] $(\varphi \land \psi) > \chi \models \varphi > (\psi > \chi)$.

\item[Monotonicity:] If $\psi\models \psi'$, then $\varphi > \psi \models \varphi > \psi'$ and $\psi' > \varphi\models \psi > \varphi$.
\end{description}

\noindent These properties are often expected of conditionals, with an  exception being the so-called \emph{counterfactual conditionals} (\cite{Lew1973,StaKoc2025}). 



We also include a conditional $\Rightarrow$, which is  an object-language version of semantic entailment ($T\models \psi \Rightarrow \chi$ iff for all $S$, $S\models\psi$ implies $S\models\chi$), and a somewhat pathological example of a van Fraassen conditional for union closed languages. 

 In the tables, ``ucl'' abrreviates ``union closed'', ``dwcl'' abbreviates ``downward closed'' and ``upcl'' abbreviates upwards closed. A ``yes'' indicates that all conditionals formed using the operator $>$ indicated in the row have the property given by the column; ``preserved'' means that the operator preserves the property, i.e. $\psi > \chi$ has the property if $\psi$ and $\chi$ have it; ``no'' means that $>$ does not even preserve the closure property.  The marker ``preserved (r)'' means that $\psi > \chi$ has the given closure property if the  consequent $\chi$  has it; the marker ``preserved (l)'' concerns the antecedent rather than the consequent. The row ``generalizes $\rightarrow$'' in Table 1 states whether $>$ coincides with the intuitionistic implication $\rightarrow$ when applied to downward closed antecedents.

 For reasons of space, the simple 
 proofs of the facts stated in the tables are omitted.
 \commf{If any is included, I would opt for the fact that $\maxcond$ preserves intersection closure (or its dual) -- the proof is quite nontrivial. Also the proof that $\ecf$ preserves union closure I find somewhat interesting.} We remark that some of the properties that we state to fail in general may hold in special cases: e.g., for $\maxcond$, strong transitivity fails in general but holds in downward closed languages.



\begin{sidewaystable}[h]
\centering
\setlength{\tabcolsep}{3pt} 
\caption{Inferential properties}
\begin{tabular}{|l|
c|c|c|c|c|c|c|}
\hline
\textbf{Conditional} 
& $\Rightarrow$ & $\ecf$ & $\wcond$ & $\bcond$ & $\maxcond$ & $\mincond$ & $\rcond$ \\
\hline

Description 
& \shortstack{entailment} 
& \shortstack{epistemic\\counterfactual} 
& \shortstack{epistemic\\indicative} 
& \shortstack{epistemic\\conditional} 
& \shortstack{maximal\\implication} 
& \shortstack{minimal\\implication} 
& \shortstack{relevant} \\
\hline

Modus Ponens 
& yes 
& \shortstack{yes for\\dwcl antecedent} 
& \shortstack{yes for\\dwcl antecedent} 
& \shortstack{yes for\\dwcl antecedent} 
& yes 
& yes 
& yes \\
\hline

Deduction Theorem 
& introduction rule 
& \shortstack{yes for\\dwcl context} 
& yes 
& \shortstack{yes for\\dwcl context} 
& \shortstack{yes for\\dwcl context} 
& \shortstack{yes for\\upcl context} 
& \shortstack{dwcl antecedent\\and context} \\
\hline

Transitivity 
& strong & strong & intermediate & strong & weak & weak & strong \\
\hline

Antecedent strengthening 
& yes & yes & yes & yes & \shortstack{for dwcl\\consequent} 
& \shortstack{for upcl\\consequent} & yes \\
\hline

Import 
& yes & yes & yes & yes & \shortstack{for dwcl\\consequent} 
& \shortstack{for upcl\\consequent} & yes \\
\hline

Export 
& no & yes & yes & yes & no & no & no \\
\hline

Monotonicity 
& yes & yes & yes & yes & \shortstack{if dwcl\\consequent} 
& \shortstack{if upcl\\consequent} & yes \\
\hline

Generalizes $\rightarrow$ 
& no & yes & no & no & \shortstack{yes if\\antec.\ local} & no & yes \\
\hline

\end{tabular}

\caption{Closure properties}
\begin{tabular}{|c|c|c|c|c|c|c|c|}
\hline
\textbf{Conditional} 
& $\Rightarrow$ 
& $\ecf$ 
& $\wcond$ 
& $\bcond$ 
& $\maxcond$ 
& $\mincond$ 
& $\rcond$ \\
\hline

\makecell{Empty team closure}
& no 
& preserved 
& preserved (r) 
& preserved (r) 
& preserved (r) 
& no 
& preserved (r) \\
\hline

Flatness
& yes except $\emptyset$ 
& \makecell{yes for ucl + $\emptyset$} 
& no 
& no 
& preserved (r) 
& no 
& preserved (l) \\
\hline

\makecell{Downward closure}
& yes 
& yes 
& no 
& no 
& preserved (r) 
& no 
& preserved (l) \\
\hline

\makecell{Union closure}
& yes 
& preserved (r) 
& preserved (r) 
& preserved (r) 
& no 
& preserved 
& yes \\
\hline

Convexity
& yes 
& yes 
& no 
& no 
& no 
& no 
& yes \\
\hline

\makecell{Upward closure}
& yes 
& no 
& preserved (r) 
& preserved (r) 
& no 
& preserved (r) 
& preserved (r) \\
\hline

\makecell{Intersection closure}
& yes 
& yes 
& no 
& no 
& preserved 
& no 
& preserved (r) \\
\hline

\end{tabular}
\end{sidewaystable}

\bibliographystyle{plain}

\end{document}